\documentclass[11pt]{amsart}

\usepackage{amsmath,amssymb,amsthm,mathtools}
\usepackage{enumitem}
\usepackage{booktabs}
\usepackage{array}
\usepackage{microtype}
\usepackage{hyperref}
\hypersetup{
  colorlinks=true,
  linkcolor=blue,
  citecolor=blue,
  urlcolor=blue,
  pdftitle={First-Order Laws for Random Geometric Graphs on the Torus},
  pdfauthor={Simi Haber, Mostafa Mirabi, and Saharon Shelah},
  pdfkeywords={random geometric graph, zero-one law, convergence law, torus, Poisson convergence, convex body}
}

\newtheorem{theorem}{Theorem}[section]

\newtheorem{lemma}[theorem]{Lemma}
\newtheorem{corollary}[theorem]{Corollary}
\theoremstyle{definition}

\theoremstyle{remark}
\newtheorem{remark}[theorem]{Remark}

\newcommand{\T}{\mathbb T}
\newcommand{\R}{\mathbb R}
\newcommand{\Z}{\mathbb Z}
\newcommand{\E}{\mathbb E}
\newcommand{\Prob}{\mathbb P}
\newcommand{\1}{\mathbf 1}
\newcommand{\Poi}{\operatorname{Poisson}}
\newcommand{\vol}{\operatorname{vol}}
\newcommand{\area}{\operatorname{area}}

\newcommand{\symdiff}{\triangle}

\title[First-Order Laws for Random Geometric Graphs on the Torus]{First-Order Laws for Random Geometric Graphs on the Torus}

\author{Simi Haber}
\address{\newline Bar-Ilan University, Ramat Gan, Israel}
\email{simi@math.biu.ac.il}

\author{Mostafa Mirabi}
\address{\newline The Taft School, Watertown, CT 06795, USA and \newline  Wesleyan University, Middletown, CT 06459, USA}
\email{mmirabi@wesleyan.edu}
\urladdr{https://sites.google.com/site/mostafamirabi}

\author{Saharon Shelah}
\address{\newline Hebrew University of Jerusalem, Jerusalem, Israel}
\email{saharon.shelah@mail.huji.ac.il}
\urladdr{https://shelah.logic.at}

\date{}

\subjclass[2020]{Primary 05C80, 03C13; Secondary 60D05, 60C05, 52A20}
\keywords{random geometric graph, first-order logic, zero-one law, Poisson convergence, torus}
\thanks{The third author research partially supported by the Israel Science Foundation
(ISF) grant no: 1838/19, and Israel Science Foundation (ISF) grant no: 2320/23; Research partially
supported by the grant “Independent Theories” NSF-BSF, (BSF 3013005232). }

\begin{document}

\begin{abstract}
Let $G_D(n;r)$ be the random geometric graph generated by $n$ independent
uniform points on the $D$-dimensional torus, with adjacency defined by
torus $L^\infty$-distance at most $r$.  We study first-order zero-one and
convergence laws at fixed radius and in sparse regimes.  At fixed radius,
we determine the asymptotics of the expected number of adjacent twin
pairs in every dimension.  In dimension two, more generally, for each
origin-symmetric convex connection body
$K\subset(-1/2,1/2)^2$, the twin count converges to a Poisson variable
with mean $\operatorname{area}(K^\circ)/16$; hence the zero-one law fails
for all fixed $0<r<1/2$ in the $L^\infty$ and Euclidean models.  At each
critical component threshold
$n^k r_n^{D(k-1)}\to a\in(0,\infty)$, the numbers of components of the
feasible connected $k$-vertex types converge jointly to independent
Poisson variables, yielding the complete first-order convergence law.
Between consecutive thresholds a zero-one law holds.  For $D\ge3$, we
also construct a definable common-neighborhood configuration of
probability order $1/n$.
\end{abstract}

\maketitle

\section{Introduction}

A sequence of random graphs satisfies a first-order zero-one law if the
probability of every first-order graph sentence tends to either zero or one.
It satisfies a first-order convergence law if each of these probabilities
has a limit, not necessarily zero or one.  The classical constant-density
zero-one law for the Erd\H{o}s--R\'enyi graph was proved independently by
Glebskii, Kogan, Liogonkii and Talanov \cite{Glebskii1969} and by Fagin
\cite{Fagin1976}.  The sparse theory was developed in particular by Shelah
and Spencer \cite{ShelahSpencer1988}; see also \cite{Libkin2004,Spencer2001}.

We study these questions for random geometric graphs.  Let
$X_1,\ldots,X_n$ be independent uniform points of the torus
$\T^D=(\R/\Z)^D$.  For $0\le r\le1/2$, let $G_D(n;r)$ be the graph on
$[n]$ in which
\[
   ij\in E(G_D(n;r))
   \quad\Longleftrightarrow\quad
   d_\infty(X_i,X_j)\le r,
\]
where $d_\infty$ is the torus $L^\infty$ metric.  The restriction
$r\le1/2$ is natural: the $L^\infty$-diameter of $\T^D$ is $1/2$, so the
graph is complete when $r\ge1/2$.

We also use a planar convex connection model.  If
$K\subset(-1/2,1/2)^2$ is an origin-symmetric convex body, let $G_K(n)$ be
the graph on $[n]$ in which $ij$ is an edge exactly when the representative
of $X_j-X_i$ in $[-1/2,1/2)^2$ belongs to $K$.  The square
$K=[-r,r]^2$ gives the planar $L^\infty$ model, while $K=rB_2$ gives the
usual Euclidean torus model for $0<r<1/2$, where $B_2$ is the unit
Euclidean disk.

Random geometric graphs originate in Gilbert's model of random plane
networks \cite{Gilbert1961}; a standard reference is Penrose
\cite{Penrose2003}.  Their logical behavior differs sharply from that of
binomial random graphs because nearby edge events are strongly dependent.
A small geometric coincidence can force many graph-theoretic relations at
once, and it is often these local coincidences that determine the limiting
first-order behavior.

McColm proved a zero-one law for the fixed-radius model on the circle
\cite{McColm1999}.  In the two-dimensional flat torus with the Euclidean
metric, Spencer and Agarwal found a first-order property with a non-trivial
limit for sufficiently small fixed radius \cite{SpencerAgarwal2006}.
Haber, Hershko and M\"uller later showed that, for sufficiently small
fixed radius, the convergence law itself fails in that model
\cite{HHM2022}.
Their proof interprets finite graph
structures and, ultimately, finite arithmetic.

Here adjacent twins give a direct obstruction to the zero-one law.  For an
arbitrary origin-symmetric planar convex connection body $K$, their number
has a Poisson limit whose mean is $\area(K^\circ)/16$.  In particular, for
the Euclidean ball $rB_2$ the mean is $\pi/(16r^2)$, so the zero-one law
fails for every $0<r<1/2$.  The same theorem gives the mean $1/(8r^2)$ for the planar $L^\infty$ model.  This does not replace the stronger non-convergence theorem of \cite{HHM2022} in the range treated there; it gives a simpler explicit zero-one obstruction and shows that the mechanism is not special to axis-parallel squares.

The same obstruction disappears in higher dimensions.  For fixed $D$, the expected number of adjacent twin pairs is of order $n^{2-D}$, and hence tends to zero when $D\ge3$.  We do not settle the fixed-radius zero-one or convergence law in these dimensions.  We do, however, prove that for every $D\ge3$ and every fixed $0<r<1/2$ there is a definable common-neighborhood configuration with probability of order $1/n$.  This is the critical scale
one would expect in an extension-based interpretation of arithmetic.  A full interpretation would require uniform joint estimates and concentration
for many such configurations; the marginal estimate proved here is not
sufficient by itself.

The sparse regimes have a different structure.  At the first edge scale the
graph is asymptotically a matching with isolated vertices, and its entire
first-order convergence law is determined by the limiting Poisson edge
count.  The same picture extends to every component threshold.  At the
$k$-vertex threshold, the counts of all strictly $D$-geometric connected
$k$-vertex component types have a joint independent Poisson limit, every
strictly $D$-geometric type of smaller order occurs many times, and larger
components disappear.  A component-wise Ehrenfeucht--Fra\"isse argument then
gives the complete first-order convergence law.  Between consecutive
component-size thresholds all feasible bounded component types occur many
times, and first-order randomness disappears.

\subsection{Main results}

We first state the planar convex-body result.  For a convex body
$K\subset\R^2$ containing the origin, write
\[
   K^\circ=\{u\in\R^2:\langle u,x\rangle\le1
             \text{ for every }x\in K\}
\]
for its polar body.

\begin{theorem}
\label{thm:planar-convex-twins}
Let $K\subset(-1/2,1/2)^2$ be an origin-symmetric convex body, and let
$T_n(K)$ be the number of adjacent twin pairs in $G_K(n)$.  Then
\begin{equation}
\label{eq:convex-twin-poisson}
   T_n(K)\xrightarrow{d}
   \Poi\left(\frac{\area(K^\circ)}{16}\right).
\end{equation}
Consequently, $G_K(n)$ does not satisfy a first-order zero-one law.
In particular,
\[
   T_n([-r,r]^2)\xrightarrow{d}\Poi\left(\frac1{8r^2}\right)
\]
and
\[
   T_n(rB_2)\xrightarrow{d}\Poi\left(\frac{\pi}{16r^2}\right)
\]
for every $0<r<1/2$.
\end{theorem}

For the $L^\infty$ model in arbitrary dimension we have the following
additional calculation.

\begin{theorem}
\label{thm:fixed-twins}
Fix $D\ge1$ and $0<r<1/2$.  Let $T_{n,D}(r)$ be the number of adjacent twin
pairs in $G_D(n;r)$.  Then
\begin{equation}
\label{eq:expected-twins}
   \E T_{n,D}(r)
   \sim 2^{D-1-D^2}r^{-D(D-1)}n^{2-D}.
\end{equation}
In particular,
\begin{equation}
\label{eq:poisson-twins}
   T_{n,2}(r)\xrightarrow{d}\Poi\left(\frac1{8r^2}\right).
\end{equation}
Consequently, $G_2(n;r)$ does not satisfy a first-order zero-one law for any
fixed $0<r<1/2$.  For every fixed $D\ge3$,
$T_{n,D}(r)\to0$ in probability.
\end{theorem}

The planar Poisson limit is therefore universal over origin-symmetric convex
connection bodies below the injectivity scale.  For the $L^\infty$ model,
adjacent twins settle the fixed-radius question in dimension two but do not
provide an obstruction in dimensions $D\ge3$.  We next turn to shrinking
radii.  For integers $j,N\ge0$, let $M(j,N)$ be the graph consisting of $j$
disjoint edges and $N$ isolated vertices.

\begin{theorem}
\label{thm:first-edge-convergence}
Fix $D\ge1$, and suppose that $n^2r_n^D\to a\in[0,\infty)$.  Put
$
   \lambda=2^{D-1}a.
$
For every first-order graph sentence $\varphi$, the limit
\[
   \lim_{n\to\infty}\Prob(G_D(n;r_n)\models\varphi)
\]
exists.  More precisely, for each $j\ge0$, let
$b_\varphi(j)\in\{0,1\}$ be the eventual truth value of $\varphi$ on
$M(j,N)$ as $N\to\infty$.  Then
\begin{equation}
\label{eq:first-edge-limit-law}
   \lim_{n\to\infty}\Prob(G_D(n;r_n)\models\varphi)
   =\sum_{j=0}^{\infty}e^{-\lambda}\frac{\lambda^j}{j!}\,b_\varphi(j).
\end{equation}
In particular, if $a=0$, then $G_D(n;r_n)$ is empty with high probability and
the zero-one law holds.  If $a\in(0,\infty)$, then the zero-one law fails.
\end{theorem}

Call a finite connected graph $H$ strictly $D$-geometric if either $H=K_1$
or, after writing $V(H)=\{1,\ldots,s\}$, there are points
$z_1,\ldots,z_s\in\R^D$ such that for every $i<j$,
\[
   ij\in E(H) \Rightarrow \|z_i-z_j\|_\infty<1,
   \qquad
   ij\notin E(H) \Rightarrow \|z_i-z_j\|_\infty>1.
\]

For $s\ge1$, let $\mathcal H_D^{(s)}$ be a fixed set of representatives
of the isomorphism classes of strictly $D$-geometric connected graphs
on exactly $s$ vertices, and put
\[
   \mathcal H_{D,k}
   =\bigcup_{1\le s\le k}\mathcal H_D^{(s)}.
\]


For a finite connected graph $H$ on $s$ vertices, let
$G_\infty(0,z_2,\ldots,z_s)$ be the graph on the displayed points in which
two points are adjacent exactly when their $L^\infty$-distance is at most
one, and define
\begin{equation}
\label{eq:general-geometric-volume}
   \gamma_D(H)
   :=
   \int_{(\R^D)^{s-1}}
   \1\{G_\infty(0,z_2,\ldots,z_s)\cong H\}
   \,dz_2\cdots dz_s.
\end{equation}
For $H=K_1$, this is the integral over the zero-dimensional space, so
$\gamma_D(K_1)=1$.  Let $C_{n,H}$ denote the number of connected components
of $G_D(n;r_n)$ isomorphic to $H$.

\begin{theorem}
\label{thm:critical-component-window}
Fix $D\ge1$ and $k\ge2$, and suppose that
\[
   n^k r_n^{D(k-1)}\longrightarrow a\in(0,\infty).
\]
For $H\in\mathcal H_D^{(k)}$, put
\[
   \lambda_H=\frac{a\gamma_D(H)}{k!}.
\]
Then the following hold.
\begin{enumerate}[label=\textup{(\roman*)}]
\item The vector
$
   \bigl(C_{n,H}:H\in\mathcal H_D^{(k)}\bigr)
$
converges jointly in distribution to a vector
$\bigl(Z_H:H\in\mathcal H_D^{(k)}\bigr)$ of independent random variables
with
\[
   Z_H\sim\Poi(\lambda_H).
\]
\item For every $H\in\mathcal H_{D,k-1}$,
$
   C_{n,H}\longrightarrow\infty
   \qquad\text{in probability}.
$4
\item With high probability every connected component has at most $k$
vertices and is isomorphic to a member of $\mathcal H_{D,k}$.
\end{enumerate}

For a vector
$\mathbf m=(m_H)_{H\in\mathcal H_D^{(k)}}$ of nonnegative integers and
$N\ge0$, let $F_{\mathbf m,N}$ be the finite graph consisting of $N$ copies
of every graph in $\mathcal H_{D,k-1}$ and exactly $m_H$ copies of $H$ for
each $H\in\mathcal H_D^{(k)}$.  For every first-order graph sentence
$\varphi$, the truth value of $F_{\mathbf m,N}\models\varphi$ is constant
for all sufficiently large $N$ by
Lemma~\ref{lem:mixed-finite-profile-logic}; denote this eventual truth value
by $b_\varphi(\mathbf m)$.  Then
\begin{equation}
\label{eq:critical-component-limit-law}
\begin{aligned}
&\lim_{n\to\infty}\Prob(G_D(n;r_n)\models\varphi)\\
&\qquad=
\sum_{\mathbf m}
   b_\varphi(\mathbf m)
   \prod_{H\in\mathcal H_D^{(k)}}
      e^{-\lambda_H}\frac{\lambda_H^{m_H}}{m_H!}.
\end{aligned}
\end{equation}
where the sum ranges over all vectors
$\mathbf m=(m_H)_{H\in\mathcal H_D^{(k)}}$ of nonnegative integers.  In
particular, $G_D(n;r_n)$ satisfies a first-order convergence law but not a
zero-one law.
\end{theorem}

The first-edge theorem is the case $k=2$ of this result when $a>0$, written
separately because it also includes the degenerate case $a=0$ and has a
particularly simple explicit formula.  The clique coordinate gives the
following earlier obstruction.

\begin{corollary}
\label{thm:clique-components}
Fix $D\ge1$ and $k\ge2$, and let $C_{n,k}:=C_{n,K_k}$.  If
\[
   n^k r_n^{D(k-1)}\longrightarrow a\in[0,\infty),
\]
then
\begin{equation}
\label{eq:clique-poisson}
   C_{n,k}\xrightarrow{d}
   \Poi\left(\frac{k^D}{k!}a\right).
\end{equation}
Consequently, if $a\in(0,\infty)$, the first-order sentence asserting the
existence of a component isomorphic to $K_k$ has limiting probability
\[
   1-\exp\left\{-\frac{k^D}{k!}a\right\},
\]
which lies strictly between zero and one.
\end{corollary}

\begin{theorem}
\label{thm:between-component-windows}
Fix $D\ge1$ and an integer $k\ge1$.  Suppose that
\[
   n^k r_n^{D(k-1)}\longrightarrow\infty,
   \qquad
   n^{k+1}r_n^{Dk}\longrightarrow0.
\]
Then $G_D(n;r_n)$ satisfies a first-order zero-one law.  More precisely, with
high probability every component has at most $k$ vertices, every component
is isomorphic to a member of $\mathcal H_{D,k}$, and the number of components
isomorphic to each $H\in\mathcal H_{D,k}$ tends to infinity in probability.
Consequently, $G_D(n;r_n)$ converges first-order to the deterministic theory
of the disjoint union containing countably infinitely many copies of every
graph in $\mathcal H_{D,k}$.
\end{theorem}

Our last main result is a fixed-radius thin-cell estimate.  The formula
$\Theta_q$ is defined in Section~\ref{sec:higher-arity}.

\begin{theorem}
\label{thm:higher-arity-thin-cells}
Let $D\ge3$ and $0<r<1/2$.  Put
$
   g=1-2r,
$ and $
   q=\left\lfloor\frac1g\right\rfloor+1.
$
There are constants $0<c_{D,r}<C_{D,r}<\infty$ such that, for all
sufficiently large $n$,
\[
   \frac{c_{D,r}}n
   \le
   \Prob\bigl(G_D(n;r)\models\Theta_q(1,\ldots,q)\bigr)
   \le
   \frac{C_{D,r}}n.
\]
\end{theorem}

The theorem gives a one-tuple probability estimate.  It does not assert
asymptotic independence for different root tuples, and therefore does not by
itself yield an interpretation of arithmetic or a failure of the convergence
law in dimensions $D\ge3$.

\begin{center}
\small
\renewcommand{\arraystretch}{1.15}
\begin{tabular}{>{\centering\arraybackslash}p{0.31\textwidth}
                >{\centering\arraybackslash}p{0.28\textwidth}
                >{\centering\arraybackslash}p{0.28\textwidth}}
\toprule
Regime & Conclusion & Main reason \\
\midrule
$D=1$, $0<r<1/2$ & zero-one law & McColm's theorem \\
Planar convex $K$ at fixed scale & zero-one law fails & Poisson adjacent twins \\
$D\ge3$, $0<r<1/2$ & open here & adjacent twins vanish \\
$n^2r_n^D\to0$ & zero-one law & empty graph with high probability \\
$n^2r_n^D\to a\in(0,\infty)$ & convergence law, not zero-one & Poisson matching \\
$n^k r_n^{D(k-1)}\to a\in(0,\infty)$ & complete convergence law, not zero-one & joint Poisson component profile \\
$n^k r_n^{D(k-1)}\to\infty$, $n^{k+1}r_n^{Dk}\to0$ & zero-one law & saturated finite components \\
$D\ge3$, $0<r<1/2$ & definable thin cells of order $1/n$ & common-neighborhood geometry \\
$r_n\ge1/2$ eventually & zero-one law & complete graph \\
\bottomrule
\end{tabular}
\end{center}

The paper is organized as follows.  Section~\ref{sec:elementary-logic}
records the elementary logical facts used in the proofs.
Section~\ref{sec:first-edge-window} treats the first edge window,
Section~\ref{sec:component-thresholds} proves the complete convergence law at
all critical component windows, and Section~\ref{sec:between-component-windows}
treats the regimes between consecutive component thresholds.
Section~\ref{sec:adjacent-twins} proves the planar convex-body Poisson theorem
and the all-dimensional $L^\infty$ twin asymptotics.
Section~\ref{sec:higher-arity} proves the thin-cell estimate.

\subsection{Notation and conventions}

Throughout the paper $D\ge1$ is fixed.  For $t\in\R/\Z$, write
\[
   \|t\|_{\T}:=\min_{m\in\Z}|t-m|.
\]
For $x,y\in\T^D$, set
\[
   d_\infty(x,y):=\max_{1\le i\le D}\|x_i-y_i\|_{\T}.
\]
The closed $L^\infty$ ball of radius $r$ around $x$ is
\[
   B_r(x):=\{y\in\T^D:d_\infty(x,y)\le r\}.
\]
For $0\le r\le1/2$,
\begin{equation}
\label{eq:ball-volume}
   \vol(B_r(x))=(2r)^D.
\end{equation}

All graphs are finite, simple and undirected, and are regarded as structures
in the first-order language with equality and one binary relation symbol
$E$.  All asymptotics are taken as $n\to\infty$, with $D$ and the other
displayed parameters fixed.  We write ``with high probability'' for
probability tending to one and $\Poi(\lambda)$ for the Poisson distribution
with mean $\lambda$.  Volume means normalized Lebesgue measure on a torus or
ordinary Lebesgue measure in a Euclidean chart.

\section{Elementary logical facts}
\label{sec:elementary-logic}

We shall use the following elementary consequence of the
Ehrenfeucht--Fra\"isse game.  It says that first-order logic cannot distinguish
between sufficiently large empty graphs, sufficiently large complete graphs, or
matchings with a fixed number of edges and sufficiently many isolated vertices.

\begin{lemma}
\label{lem:empty-complete-matching}
For every first-order graph sentence $\varphi$, the following eventual truth
values exist.
\begin{enumerate}[label=\textup{(\roman*)}]
\item The truth value of $\varphi$ on the empty graph with $N$ vertices is
constant for all sufficiently large $N$.
\item The truth value of $\varphi$ on the complete graph with $N$ vertices is
constant for all sufficiently large $N$.
\item For each fixed $j\ge0$, the truth value of $\varphi$ on $M(j,N)$ is
constant for all sufficiently large $N$.
\end{enumerate}
\end{lemma}

\begin{proof}
Let $q$ be the quantifier depth of $\varphi$.  We prove that the relevant
graphs are equivalent for all first-order sentences of quantifier depth at most
$q$ once $N$ is large enough.

First consider two empty graphs on $N$ and $N'$ vertices, with
$N,N'\ge q$.  Duplicator wins the $q$-round Ehrenfeucht--Fra\"isse game by
maintaining a partial bijection between the vertices already chosen in the two
graphs.  If Spoiler repeats a previously chosen vertex, Duplicator repeats the
corresponding vertex.  If Spoiler chooses a new vertex, Duplicator chooses a
new vertex in the other graph.  Since at most $q$ distinct vertices can be
chosen during the game, this is always possible.  The resulting partial map
preserves equality and, because both graphs are empty, also preserves
adjacency.  Hence all empty graphs with at least $q$ vertices satisfy the same
sentences of quantifier depth at most $q$.

The same argument applies to complete graphs.  In that case any partial
bijection preserves equality, and it also preserves adjacency because two
distinct vertices are adjacent in both graphs.

It remains to treat $M(j,N)$ for fixed $j$.  Let
$N,N'\ge q$.  We compare $M(j,N)$ and $M(j,N')$.  Match the $j$ two-vertex edge
components in the two graphs in advance, and fix an isomorphism between each
matched pair of edge components.  During the game, Duplicator maintains the
following invariant.  Vertices chosen in edge components are mapped according
to these fixed component isomorphisms, and isolated vertices chosen so far are
paired bijectively with isolated vertices in the other graph.

Suppose Spoiler plays in an edge component.  If that component has already been
entered during the play, Duplicator answers with the corresponding vertex in
the matched component.  If it has not yet been entered, Duplicator uses the
preassigned isomorphism from that edge component to its matched edge component.
Suppose instead that Spoiler plays an isolated vertex.  If this isolated vertex
has already been chosen, Duplicator repeats the previously paired isolated
vertex.  If it is new, Duplicator chooses a new isolated vertex in the other
graph and pairs the two.  Since there are at most $q$ rounds and
$N,N'\ge q$, there is always an unused isolated vertex available when one is
needed.

After each round, the chosen vertices in the two graphs are related by a
partial isomorphism: equality is preserved by construction, and adjacency is
preserved because the only edges in $M(j,N)$ and $M(j,N')$ lie inside the
matched two-vertex components.  Thus Duplicator wins the $q$-round game.

Therefore, for each fixed $j$, all graphs $M(j,N)$ with $N\ge q$ satisfy the
same first-order sentences of quantifier depth at most $q$.  Applying this to
$\varphi$ proves the three asserted eventual truth values.
\end{proof}

For later use we record a component-wise version of the same argument.  Let
$\mathcal S$ and $\mathcal T$ be disjoint finite sets of finite connected
graphs.  For a vector $\mathbf m=(m_H)_{H\in\mathcal T}$ of nonnegative
integers, let $\mathfrak C_q(\mathcal S,\mathcal T;\mathbf m)$ be the class of finite
graphs whose components are all isomorphic to members of
$\mathcal S\cup\mathcal T$, which have at least $q$ components of every type
in $\mathcal S$, and which have exactly $m_H$ components isomorphic to $H$
for every $H\in\mathcal T$.

\begin{lemma}
\label{lem:mixed-finite-profile-logic}
Let $\varphi$ be a first-order graph sentence of quantifier depth $q$.  For
every fixed $\mathbf m$, all graphs in
$\mathfrak C_q(\mathcal S,\mathcal T;\mathbf m)$ have the same truth value
for $\varphi$.
\end{lemma}

\begin{proof}
Let $G,G'\in\mathfrak C_q(\mathcal S,\mathcal T;\mathbf m)$.  For each
$H\in\mathcal T$, match the $m_H$ components of type $H$ in $G$ bijectively
with those in $G'$, and fix an isomorphism on each matched pair.  During the
$q$-round Ehrenfeucht--Fra\"isse game, whenever Spoiler first enters a
component of a type in $\mathcal S$, Duplicator chooses a previously unused
component of the same type in the other graph and fixes an isomorphism
between them.  Whenever Spoiler enters a component of a type in
$\mathcal T$, Duplicator uses its preassigned mate.  All later moves in an
entered component are answered through the fixed component isomorphism.

At most $q$ components can be entered during the game, and each graph has at
least $q$ components of every type in $\mathcal S$.  Thus a fresh component
of the required type is always available.  The chosen vertices form a
partial isomorphism after every round, since adjacency is preserved inside
matched components and no edges run between different components.  Hence
Duplicator wins the $q$-round game.
\end{proof}

\section{The first edge window}
\label{sec:first-edge-window}

Let
$
   Z_n=e(G_D(n;r_n)),
$ and $
   p_n=\Prob(d_\infty(X_1,X_2)\le r_n)=(2r_n)^D .
$

\begin{lemma}
\label{lem:edge-poisson}
If $n^2r_n^D\to a\in[0,\infty)$, then
$
   Z_n\xrightarrow{d}\Poi(2^{D-1}a).
$
Moreover, with high probability $G_D(n;r_n)$ is a matching together with
isolated vertices.
\end{lemma}

\begin{proof}
We use factorial moments.  Fix $m\ge1$.  The factorial moment $(Z_n)_m$ is the
sum over ordered $m$-tuples of distinct unordered vertex pairs.

First consider tuples in which the $m$ pairs are vertex-disjoint.  The
corresponding edge events depend on disjoint sets of sample points, and hence
are independent.  The number of ordered $m$-tuples of vertex-disjoint unordered
pairs is
$
   \frac{(n)_{2m}}{2^m}.
$
Their contribution to $\E[(Z_n)_m]$ is therefore
\[
   \frac{(n)_{2m}}{2^m}p_n^m
   \longrightarrow
   (2^{D-1}a)^m .
\]

It remains to show that tuples with overlapping pairs contribute negligibly.
Fix such a tuple, and let $H$ be the graph spanned by its selected pairs.  Let
$v$ be the number of vertices of $H$ and $c$ the number of connected components.
Since at least one component contains two selected edges, while every component
contains at least one edge, we have
$
   v>2c .
$
Choose a spanning forest of $H$.  It has $v-c$ edges.  If all selected edges are
present, then in particular all edges of this spanning forest are present.  By
exposing the vertices of the forest component by component, starting from one
root in each component, we get
\[
   \Prob(\text{all selected edges occur})\le p_n^{\,v-c}.
\]
Indeed, each non-root vertex in the forest must fall inside an $L^\infty$ ball
of radius $r_n$ around its parent, an event of probability $p_n$, conditionally
on the vertices already exposed.

For each fixed isomorphism type of $H$, there are $O(n^v)$ choices of the
involved labelled vertices.  Since $n^2r_n^D=O(1)$, we have $p_n=O(n^{-2})$.
Thus the total contribution of this type is
\[
   O(n^v p_n^{\,v-c})
   =
   O(n^{2c-v})
   =
   o(1).
\]
There are only finitely many such types for fixed $m$, so all overlapping
tuples together contribute $o(1)$.  Hence
\[
   \E[(Z_n)_m]\longrightarrow (2^{D-1}a)^m
   \qquad (m=1,2,\ldots).
\]
These are the factorial moments of $\Poi(2^{D-1}a)$, and the Poisson
convergence follows.

We now prove the matching assertion.  Let $Y_n$ be the number of unordered
pairs of distinct edges sharing a vertex.  For each triple of vertices there
are three possible pairs of edges sharing a vertex, and each such pair occurs
with probability $p_n^2$.  Therefore
\[
   \E Y_n
   \le
   3\binom n3 p_n^2
   =
   O(n^3r_n^{2D})
   =
   O(n^{-1}),
\]
because $r_n^D=O(n^{-2})$.  By Markov's inequality, $Y_n\to0$ in probability.
Thus, with high probability, no two edges share a vertex, which is exactly to
say that the graph is a matching together with isolated vertices.
\end{proof}

\begin{proof}[Proof of Theorem~\ref{thm:first-edge-convergence}]
Put
$
   \lambda=2^{D-1}a.
$
By Lemma~\ref{lem:edge-poisson}, $Z_n\xrightarrow{d}\Poi(\lambda)$, and with
high probability $G_D(n;r_n)$ is a matching together with isolated vertices.

Let $\varphi$ be a first-order graph sentence.  For $j\ge0$, let
$b_\varphi(j)$ be the eventual truth value of $\varphi$ on $M(j,N)$ as
$N\to\infty$.  This eventual value exists by
Lemma~\ref{lem:empty-complete-matching}.

Let $A_n$ be the event that $G_D(n;r_n)$ is a matching.  Fix $J\ge0$.  For all
sufficiently large $n$, and for every $0\le j\le J$, the graph
$M(j,n-2j)$ satisfies $\varphi$ if and only if $b_\varphi(j)=1$.  Hence, for
such $n$,
\[
\left|
   \Prob(G_D(n;r_n)\models\varphi)
   -
   \sum_{j=0}^{J} b_\varphi(j)\Prob(Z_n=j)
\right|
\le
   \Prob(A_n^c)+\Prob(Z_n>J).
\]
The first term tends to zero.  Since $Z_n$ converges in distribution to a
Poisson random variable, the sequence $(Z_n)$ is tight.  Given
$\varepsilon>0$, choose $J$ so large that
$
   \limsup_{n\to\infty}\Prob(Z_n>J)<\varepsilon .
$
Letting $n\to\infty$ with this fixed $J$ gives
\[
\limsup_{n\to\infty}
\left|
   \Prob(G_D(n;r_n)\models\varphi)
   -
   \sum_{j=0}^{J}
      e^{-\lambda}\frac{\lambda^j}{j!}\,b_\varphi(j)
\right|
\le \varepsilon .
\]
Finally let $J\to\infty$.  Since $b_\varphi(j)\in\{0,1\}$, the Poisson series
is absolutely convergent, and we obtain
\[
   \lim_{n\to\infty}\Prob(G_D(n;r_n)\models\varphi)
   =
   \sum_{j=0}^{\infty}
      e^{-\lambda}\frac{\lambda^j}{j!}\,b_\varphi(j),
\]
which is the claimed convergence law.

If $a=0$, then $\lambda=0$, so $Z_n\to0$ in probability.  Thus
$G_D(n;r_n)$ is empty with high probability, and the zero-one law follows from
Lemma~\ref{lem:empty-complete-matching}.  If $a\in(0,\infty)$, the first-order
sentence
$
   \exists x\exists y\,(x\ne y\wedge E(x,y))
$
has limiting probability
$
   1-e^{-\lambda}=1-e^{-2^{D-1}a},
$
which lies strictly between $0$ and $1$.  Hence the zero-one law fails in this
case.
\end{proof}

\begin{remark}
The condition $r_n^D\ll n^{-1}$ says that the expected degree tends to zero; it
does not by itself imply that the graph is empty with high probability.  The
empty-graph regime is
$
   n^2r_n^D\to0,
$
since
\[
   \E e(G_D(n;r_n))
   =
   \binom n2(2r_n)^D
   \sim
   2^{D-1}n^2r_n^D .
\]
\end{remark}

\section{Critical sparse component windows}
\label{sec:component-thresholds}

This section proves Theorem~\ref{thm:critical-component-window} and then
deduces Corollary~\ref{thm:clique-components}.  We begin with two geometric
facts about the constants in \eqref{eq:general-geometric-volume}.

\begin{lemma}
\label{lem:general-geometric-volume}
If $H$ is a finite connected graph, then $\gamma_D(H)<\infty$.  If $H$ is
strictly $D$-geometric, then $\gamma_D(H)>0$.
\end{lemma}

\begin{proof}
Let $s=|V(H)|$ and suppose that
$G_\infty(0,z_2,\ldots,z_s)\cong H$.  Since this graph is connected, every
$z_i$ can be joined to $0$ by a path with at most $s-1$ edges.  Hence
\[
   \|z_i\|_\infty\le s-1
   \qquad(2\le i\le s),
\]
so the region in \eqref{eq:general-geometric-volume} is bounded and has
finite volume.

Now suppose that $H$ is strictly $D$-geometric.  Choose a strict realization
and translate the point carrying label $1$ to the origin.  All edge
distances are strictly less than one and all non-edge distances are strictly
greater than one.  These inequalities persist on a sufficiently small open
neighborhood of the realization.  This neighborhood is contained in the
region defining $\gamma_D(H)$ and has positive volume.
\end{proof}

\begin{lemma}
\label{lem:general-component-probability}
Let $H$ be a fixed connected graph on $s\ge1$ vertices.  If $r_n\to0$, then,
for all sufficiently large $n$,
\begin{equation}
\label{eq:induced-H-probability}
   \Prob\bigl(G_D(n;r_n)[\{1,\ldots,s\}]\cong H\bigr)
   =\gamma_D(H)r_n^{D(s-1)}.
\end{equation}
If in addition $nr_n^D\to0$, then
\begin{equation}
\label{eq:H-component-probability}
\begin{aligned}
&\Prob\bigl(\{1,\ldots,s\}
   \text{ is a component isomorphic to }H\bigr)\\
&\qquad=\gamma_D(H)r_n^{D(s-1)}(1+o(1)).
\end{aligned}
\end{equation}

More generally, let $H_1,\ldots,H_m$ be fixed strictly $D$-geometric
connected graphs, write $s_i=|V(H_i)|$, and let
$S_1,\ldots,S_m$ be pairwise disjoint fixed vertex sets with
$|S_i|=s_i$.  If $r_n\to0$ and $nr_n^D\to0$, then
\begin{equation}
\label{eq:several-components-probability}
\begin{aligned}
&\Prob\bigl(S_i\text{ is a component isomorphic to }H_i
   \text{ for every }1\le i\le m\bigr)\\
&\qquad=
\left(\prod_{i=1}^m\gamma_D(H_i)\right)
 r_n^{D\sum_{i=1}^m(s_i-1)}(1+o(1)).
\end{aligned}
\end{equation}
The error term is uniform over the choice of the pairwise disjoint sets
$S_1,\ldots,S_m$.
\end{lemma}

\begin{proof}
The assertion is immediate for $s=1$ at the induced-graph level, since
$\gamma_D(K_1)=1$.  Assume first that $s\ge2$.  By translation invariance,
condition on $X_1$ and translate it to the origin.  On the event that the
induced graph on $\{1,\ldots,s\}$ is connected, every $X_i$ is joined to
$X_1$ by a path of at most $s-1$ edges.  Therefore
\[
   d_\infty(X_i,X_1)\le(s-1)r_n.
\]
For all sufficiently large $n$, choose the unique representative
$h_i\in[-1/2,1/2)^D$ of $X_i-X_1$ satisfying
\[
   \|h_i\|_\infty=d_\infty(X_i,X_1)\le(s-1)r_n<\frac14.
\]
Then $\|h_i-h_j\|_\infty<1/2$ for all $i,j$, so
\[
   d_\infty(X_i,X_j)=\|h_i-h_j\|_\infty.
\]
Writing $h_i=r_nz_i$, the induced graph is isomorphic to $H$ exactly when
$
   (z_2,\ldots,z_s)
$
belongs to the region in \eqref{eq:general-geometric-volume}.

Conversely, if
$G_\infty(0,z_2,\ldots,z_s)\cong H$, connectedness gives
$\|z_i\|_\infty\le s-1$.  Thus, for all sufficiently large $n$, the points
$X_1+r_nz_i$ lie in the same injective torus chart, and all relevant torus
distances agree with the scaled Euclidean distances.  The change of
variables $h_i=r_nz_i$ is therefore a bijection between the two configuration
regions, with Jacobian $r_n^{D(s-1)}$.  This proves
\eqref{eq:induced-H-probability}.

On the induced-$H$ event, let
\[
   U=\bigcup_{i=1}^s B_{r_n}(X_i).
\]
The displayed vertices form an isolated component precisely when none of the
remaining sample points lies in $U$.  Uniformly over the displayed
configuration,
$
   \vol(U)\le s(2r_n)^D,
$
and hence, when $nr_n^D\to0$,
\[
   (1-\vol(U))^{n-s}=1+o(1).
\]
This proves \eqref{eq:H-component-probability}, including the case $s=1$.

We now consider several disjoint sets.  Their internal induced-graph events
are independent, and by \eqref{eq:induced-H-probability} their joint
probability is
\[
   \left(\prod_{i=1}^m\gamma_D(H_i)\right)
   r_n^{D\sum_i(s_i-1)}.
\]
Choose one anchor vertex in each $S_i$.  Conditional on the internal relative
configurations, the anchors remain independent and uniform on the torus.  By
connectedness, every vertex in $S_i$ lies within $(s_i-1)r_n$ of its anchor.
Thus, if a cross-edge joins $S_i$ to $S_j$, the two anchors must be within
$O(r_n)$ of one another.  Uniformly over the internal configurations,
\[
   \Prob(\text{some cross-edge occurs}\mid\text{internal configurations})
   =O(r_n^D).
\]
The union of the $r_n$-balls around all displayed vertices has volume
$O(r_n^D)$, so the conditional probability that some outside vertex is
adjacent to a displayed vertex is $O(nr_n^D)$.  It follows that, conditional
on all internal events, the probability that the displayed sets are distinct
isolated components is
\[
   1+O(r_n^D+nr_n^D)=1+o(1).
\]
This proves \eqref{eq:several-components-probability}.  Exchangeability makes
the estimate uniform over the chosen disjoint label sets.
\end{proof}

\begin{proof}[Proof of Theorem~\ref{thm:critical-component-window}]
Put
$
   t_n=nr_n^D.
$
The hypothesis is equivalent to
\[
   nt_n^{k-1}\longrightarrow a.
\]
Consequently,
\[
   t_n\longrightarrow0,
   \qquad r_n\longrightarrow0,
   \qquad
   nt_n^k=(nt_n^{k-1})t_n\longrightarrow0.
\]
In particular, $nr_n^D\to0$ and
\begin{equation}
\label{eq:no-k-plus-one-scale}
   n^{k+1}r_n^{Dk}\longrightarrow0.
\end{equation}

We first prove the joint Poisson limit.  Enumerate
$
   \mathcal H_D^{(k)}=\{H_1,\ldots,H_\ell\}
$
and fix nonnegative integers $m_1,\ldots,m_\ell$.  Put
$M=m_1+\cdots+m_\ell$.  Expanding the mixed factorial moment gives
\[
   \E\left[\prod_{j=1}^\ell(C_{n,H_j})_{m_j}\right].
\]
Two distinct selected components cannot overlap, so only families of $M$
pairwise disjoint $k$-sets contribute.  The number of ordered such families,
with the prescribed type attached to each position, is
$
   \frac{(n)_{kM}}{(k!)^M}.
$
By Lemma~\ref{lem:general-component-probability}, the probability that a
fixed family gives components of the prescribed types is
\[
   \left(\prod_{j=1}^\ell\gamma_D(H_j)^{m_j}\right)
   r_n^{D(k-1)M}(1+o(1)).
\]
Therefore
\begin{align*}
   \E\left[\prod_{j=1}^\ell(C_{n,H_j})_{m_j}\right]
   &\longrightarrow
   \prod_{j=1}^\ell
   \left(\frac{a\gamma_D(H_j)}{k!}\right)^{m_j}\\
   &=\prod_{j=1}^\ell\lambda_{H_j}^{m_j}.
\end{align*}
These are the mixed factorial moments of independent Poisson variables with
means $\lambda_{H_1},\ldots,\lambda_{H_\ell}$.  The multivariate factorial
moment criterion gives
\[
   (C_{n,H_1},\ldots,C_{n,H_\ell})
   \xrightarrow{d}
   (Z_{H_1},\ldots,Z_{H_\ell}),
\]
with independent $Z_{H_j}\sim\Poi(\lambda_{H_j})$.

Next fix $H\in\mathcal H_D^{(s)}$ with $1\le s<k$.  By
Lemma~\ref{lem:general-component-probability},
\begin{align*}
   \mu_{n,H}:=\E C_{n,H}
   &\sim
   \frac{\gamma_D(H)}{s!}
   n^s r_n^{D(s-1)}\\
   &=\frac{\gamma_D(H)}{s!}nt_n^{s-1}.
\end{align*}
Since $H$ is strictly $D$-geometric,
Lemma~\ref{lem:general-geometric-volume} gives $\gamma_D(H)>0$, and
\[
   nt_n^{s-1}
   =\frac{nt_n^{k-1}}{t_n^{k-s}}
   \longrightarrow\infty.
\]
Thus $\mu_{n,H}\to\infty$.  Applying
\eqref{eq:several-components-probability} to two disjoint $s$-sets gives
\begin{align*}
   \E[(C_{n,H})_2]
   &=\frac{(n)_{2s}}{(s!)^2}
     \gamma_D(H)^2r_n^{2D(s-1)}(1+o(1))\\
   &=(1+o(1))\mu_{n,H}^2.
\end{align*}
Hence
\[
   \operatorname{Var}(C_{n,H})
   =o(\mu_{n,H}^2)+O(\mu_{n,H}),
\]
and Chebyshev's inequality yields
\[
   \frac{C_{n,H}}{\mu_{n,H}}\longrightarrow1
   \qquad\text{in probability}.
\]
In particular, $C_{n,H}\to\infty$ in probability.

We now rule out larger components.  If a component has more than $k$
vertices, then some set of $k+1$ vertices spans a connected graph and hence
contains a spanning tree.  For a fixed labelled tree on $k+1$ vertices,
exposing the vertices along the tree gives probability at most
\[
   (2r_n)^{Dk}
\]
that all tree edges are present.  Since there are only finitely many labelled
trees on $k+1$ vertices, the expected number of connected $(k+1)$-sets is
\[
   O(n^{k+1}r_n^{Dk})=o(1)
\]
by \eqref{eq:no-k-plus-one-scale}.  Thus, with high probability, every
component has at most $k$ vertices.

Since the sample distribution is continuous, with probability one no pair
of sampled points has torus $L^\infty$-distance exactly $r_n$.  On the event
that every component has at most $k$ vertices, lift each component to
$\R^D$ by choosing one vertex as a root and following a spanning tree.  For
all sufficiently large $n$, the lifted component lies in a single torus
chart, and its Euclidean $L^\infty$-distances agree with the torus distances.
After scaling by $1/r_n$, every edge has distance strictly less than one and
every non-edge has distance strictly greater than one.  Hence every component
type that occurs belongs to $\mathcal H_{D,k}$.

It remains to prove the first-order limit formula.  Let $\varphi$ have
quantifier depth $q$, and let
\[
   \mathbf C_n=(C_{n,H})_{H\in\mathcal H_D^{(k)}}.
\]
Let $A_n$ be the event that every component belongs to
$\mathcal H_{D,k}$ and that there are at least $q$ components of every type
in $\mathcal H_{D,k-1}$.  The preceding arguments and the finiteness of
$\mathcal H_{D,k-1}$ give
$
   \Prob(A_n)\longrightarrow1.
$
On $A_n$, both $G_D(n;r_n)$ and $F_{\mathbf C_n,q}$ belong to
\[
   \mathfrak C_q\bigl(
      \mathcal H_{D,k-1},\mathcal H_D^{(k)};\mathbf C_n
   \bigr).
\]
Lemma~\ref{lem:mixed-finite-profile-logic} therefore gives
\[
   G_D(n;r_n)\models\varphi
   \quad\Longleftrightarrow\quad
   F_{\mathbf C_n,q}\models\varphi
\]
on $A_n$.  By the definition of $b_\varphi$,
\[
   \left|
   \Prob(G_D(n;r_n)\models\varphi)
   -\E[b_\varphi(\mathbf C_n)]
   \right|
   \le\Prob(A_n^c)\longrightarrow0.
\]

The vector $\mathbf C_n$ converges in distribution to
$\mathbf Z=(Z_H)_{H\in\mathcal H_D^{(k)}}$.  Since this is a finite vector
taking values in a countable discrete space, finite-box truncation gives
\[
   \E[b_\varphi(\mathbf C_n)]
   \longrightarrow
   \E[b_\varphi(\mathbf Z)].
\]
Independence of the Poisson coordinates now gives exactly
\eqref{eq:critical-component-limit-law}.

Finally, $K_k\in\mathcal H_D^{(k)}$ and, by
Lemma~\ref{lem:gamma-value} below, $\gamma_D(K_k)=k^D>0$.  The sentence
asserting the existence of a $K_k$-component therefore has limiting
probability
\[
   1-\exp\left\{-\frac{k^Da}{k!}\right\}\in(0,1).
\]
Thus the zero-one law fails.
\end{proof}

For completeness, we evaluate the geometric volume for cliques.  Put
$z_1=0$.  Then
\[
   \gamma_D(K_k)
   =
   \int_{(\R^D)^{k-1}}
   \1\left\{
      \|z_i-z_j\|_\infty\le1
      \text{ for all }1\le i<j\le k
   \right\}
   dz_2\cdots dz_k.
\]

\begin{lemma}
\label{lem:gamma-value}
For every $D\ge1$ and $k\ge2$,
\[
   \gamma_D(K_k)=k^D.
\]
\end{lemma}

\begin{proof}
The defining region factors over the $D$ coordinate directions.  It is
therefore enough to compute the one-dimensional volume
\[
\begin{aligned}
   v_k:=\vol\Bigl\{(u_2,\ldots,u_k)\in\R^{k-1}:\ &
   \max(0,u_2,\ldots,u_k)\\
   &{}-\min(0,u_2,\ldots,u_k)\le1\Bigr\}.
\end{aligned}
\]
Outside a null set, exactly one of the $k$ labelled numbers
$0,u_2,\ldots,u_k$ is the minimum.  If the minimum is $0$, then every $u_i$
lies in $[0,1]$, a region of volume one.  If the minimum is $u_j$, then
$u_j\in[-1,0]$ and, for each fixed $u_j$, every other $u_i$ lies in
$[u_j,u_j+1]$.  This region also has volume one after integration over
$u_j\in[-1,0]$.  The $k$ possible choices of the minimum therefore give
$v_k=k$.  Taking the product over the $D$ coordinates yields
$\gamma_D(K_k)=v_k^D=k^D$.
\end{proof}

\begin{proof}[Proof of Corollary~\ref{thm:clique-components}]
If $a\in(0,\infty)$, this is the $K_k$ coordinate of
Theorem~\ref{thm:critical-component-window}, together with
Lemma~\ref{lem:gamma-value}.  If $a=0$, then $r_n\to0$ and
$nr_n^D\to0$.  Lemma~\ref{lem:general-component-probability} gives
\[
   \E C_{n,k}
   \sim\frac{k^D}{k!}n^k r_n^{D(k-1)}\longrightarrow0.
\]
Thus $C_{n,k}\to0$ in probability by Markov's inequality, which is the
Poisson limit with mean zero.  The first-order conclusion follows from the
sentence displayed in the proof below.

The existence of a connected component isomorphic to $K_k$ is expressed by
\[
\begin{aligned}
\exists x_1\cdots\exists x_k\bigg[&
   \bigwedge_{1\le i<j\le k}x_i\ne x_j
   \wedge
   \bigwedge_{1\le i<j\le k}E(x_i,x_j) \\
&\wedge
   \forall y\left(
      \bigvee_{i=1}^k y=x_i
      \vee
      \bigwedge_{i=1}^k \neg E(y,x_i)
   \right)
\bigg].
\end{aligned}
\]
Its limiting probability for $a>0$ is
$1-\exp\{-k^Da/k!\}$.
\end{proof}

\section{Between consecutive sparse component windows}
\label{sec:between-component-windows}

We now prove Theorem~\ref{thm:between-component-windows}.  The result gives a
zero-one law in the sparse regimes between two consecutive component-size
thresholds.  For example, when $k=2$, the assumptions
\[
   n^2r_n^D\to\infty,
   \qquad
   n^3r_n^{2D}\to0
\]
imply that, with high probability, every component has size at most two, while
the numbers of edge components and isolated vertices both tend to infinity.

\begin{lemma}
\label{lem:between-window-components}
Fix $D\ge1$ and $k\ge1$, and assume
\[
   n^k r_n^{D(k-1)}\to\infty,
   \qquad
   n^{k+1}r_n^{Dk}\to0.
\]
Then, with high probability, no connected component of $G_D(n;r_n)$ has more
than $k$ vertices.  Moreover, for every strictly $D$-geometric connected graph
$H$ with $1\le |V(H)|\le k$, the number of components isomorphic to $H$ tends
to infinity in probability.
\end{lemma}

\begin{proof}
Put
$
   t_n=n r_n^D.
$ 
The second assumption says $nt_n^k\to0$, and hence $t_n\to0$.  The first
assumption says $nt_n^{k-1}\to\infty$.  Consequently,
\[
   r_n\to0,
   \qquad nr_n^D=t_n\to0,
\]
and, for every $1\le s\le k$,
\begin{equation}
\label{eq:between-all-smaller-scales}
   n^s r_n^{D(s-1)}
   =nt_n^{s-1}
   \longrightarrow\infty.
\end{equation}
Indeed, for $s<k$,
\[
   nt_n^{s-1}
   =\frac{nt_n^{k-1}}{t_n^{k-s}}\longrightarrow\infty,
\]
while the case $s=k$ is the first assumption.

We first rule out large components.  If a component has more than $k$
vertices, then some set of $k+1$ vertices spans a connected graph and hence
contains a spanning tree.  For a fixed labelled tree on $k+1$ vertices,
exposing the vertices along the tree gives probability at most
$
   (2r_n)^{Dk}
$
that all tree edges are present.  Since there are only finitely many labelled
trees on $k+1$ vertices, the expected number of connected $(k+1)$-sets is
\[
   O(n^{k+1}r_n^{Dk})=o(1).
\]
Thus, with high probability, no component has more than $k$ vertices.

Now fix $H\in\mathcal H_D^{(s)}$ with $1\le s\le k$, and put
$\mu_{n,H}=\E C_{n,H}$.  Lemmas~\ref{lem:general-geometric-volume} and
\ref{lem:general-component-probability}, together with
\eqref{eq:between-all-smaller-scales}, give
\[
   \mu_{n,H}
   \sim\frac{\gamma_D(H)}{s!}n^s r_n^{D(s-1)}
   \longrightarrow\infty.
\]
The two-component estimate \eqref{eq:several-components-probability} gives
\begin{align*}
   \E[(C_{n,H})_2]
   &=\frac{(n)_{2s}}{(s!)^2}
     \gamma_D(H)^2r_n^{2D(s-1)}(1+o(1))\\
   &=(1+o(1))\mu_{n,H}^2.
\end{align*}
Therefore
\[
   \operatorname{Var}(C_{n,H})
   =o(\mu_{n,H}^2)+O(\mu_{n,H}),
\]
and Chebyshev's inequality gives
\[
   C_{n,H}\longrightarrow\infty
   \qquad\text{in probability}.
\]
\end{proof}

\begin{proof}[Proof of Theorem~\ref{thm:between-component-windows}]
Fix $m\ge1$.  By Lemma~\ref{lem:between-window-components}, with high
probability no component has more than $k$ vertices, and every strictly
$D$-geometric connected graph with at most $k$ vertices appears as a component
at least $m$ times.

It remains to justify that, with high probability, no other component type can
occur.  Since the sample distribution is continuous, with probability one no
pair of sampled points has torus $L^\infty$-distance exactly $r_n$.  Also
$r_n\to0$.  On the high-probability event that all components have size at most
$k$, take any component $C$ with vertices $v_1,\ldots,v_s$, where $s\le k$.
If $s=1$, then $C\cong K_1\in\mathcal H_{D,k}$.  Suppose $s\ge2$.  Then
$k\ge2$, and the first hypothesis implies that $r_n>0$ for all sufficiently
large $n$.

Choose a spanning tree of $C$ and lift the component to $\R^D$ by following the
tree from a chosen root.  For all sufficiently large $n$, the whole lifted
component lies in a cube of side $O(k r_n)$, which is smaller than $1/2$ in
each coordinate.  Therefore the Euclidean $L^\infty$-distances in this lift
agree with the corresponding torus distances among the vertices of $C$.
Since no distance is exactly $r_n$, edges of $C$ have lifted distance strictly
less than $r_n$, while non-edges have lifted distance strictly greater than
$r_n$.  Scaling the lift by $1/r_n$ gives a strict $D$-geometric realization of
the isomorphism type of $C$.  Hence every component type that occurs belongs to
$\mathcal H_{D,k}$.

Let $\varphi$ be a first-order graph sentence of quantifier depth $q$.  With
probability tending to one, every component belongs to $\mathcal H_{D,k}$ and
there are at least $q$ components of every type in $\mathcal H_{D,k}$.  Taking
$\mathcal S=\mathcal H_{D,k}$ and $\mathcal T=\varnothing$ in
Lemma~\ref{lem:mixed-finite-profile-logic}, all graphs with this property have
the same truth value for $\varphi$.  Hence
\[
   \Prob(G_D(n;r_n)\models\varphi)
\]
tends to either $0$ or $1$.  This proves the zero-one law.

The same component-wise Ehrenfeucht--Fra\"isse argument identifies the
limiting theory.  For each quantifier depth $q$, the random graph is, with
high probability, $q$-equivalent to the countable disjoint union containing
infinitely many copies of every graph in $\mathcal H_{D,k}$.  Therefore
$G_D(n;r_n)$ converges first-order to the deterministic theory of that
countable disjoint union.
\end{proof}

\section{Adjacent twin pairs at fixed radius}
\label{sec:adjacent-twins}

Two adjacent vertices $x,y$ are adjacent twins if they have the same neighbors
outside the pair.  The existence of such a pair is expressed by the first-order
sentence
\begin{equation}
\label{eq:twin-sentence}
 \exists x\exists y\left[
 x\ne y\wedge E(x,y)\wedge
 \forall z\bigl((z=x\vee z=y)\vee(E(x,z)\leftrightarrow E(y,z))\bigr)
 \right].
\end{equation}

\subsection{Planar convex connection bodies}

Let $K\subset(-1/2,1/2)^2$ be an origin-symmetric convex body.  We view
$K$ as a subset of $\T^2$ and write
\[
   D_K(h):=K\symdiff(K+h),
   \qquad
   A_K(h):=\area(D_K(h)),
   \qquad h\in\T^2.
\]
Here $K+h$ is a torus translate.  For $h\in\T^2$, put
\[
   |h|_{\T,2}:=\min_{z\in\Z^2}\|h-z\|_2.
\]
Let
\[
   h_K(u):=\sup_{x\in K}\langle u,x\rangle
\]
be the support function of $K$, and let $R$ denote rotation through
$\pi/2$.

The next lemma gives the local form of the covariogram at the origin.  The
relation with projection bodies is classical; see, for example,
\cite{LangharstRoysdonZvavitch2022}.  We include the short planar argument.

\begin{lemma}
\label{lem:convex-symdiff}
For every $u\in\R^2$,
\begin{equation}
\label{eq:convex-symdiff-limit}
   A_K(tu)=4t\,h_K(Ru)+o(t)
   \qquad(t\downarrow0).
\end{equation}
Moreover, there is a constant $c_K>0$ such that
\begin{equation}
\label{eq:convex-global-lower}
   A_K(h)\ge c_K|h|_{\T,2}
   \qquad(h\in\T^2).
\end{equation}
\end{lemma}

\begin{proof}
Fix $u\ne0$, write $u=s\theta$ with $s=\|u\|_2$ and
$\|\theta\|_2=1$, and use coordinates parallel to $\theta$ and
$\theta^\perp$.  For $y\in\theta^\perp$, the section
\[
   I_y:=\{a\in\R:y+a\theta\in K\}
\]
is an interval, possibly empty; write $\ell(y)$ for its length.  Since $K$
is compactly contained in a fundamental square, for all sufficiently small
$t$ the torus and Euclidean translations agree.  On each line parallel to
$\theta$,
\[
   \operatorname{length}\bigl(I_y\symdiff(I_y+ts)\bigr)
   =2\min\{ts,\ell(y)\}.
\]
Fubini's theorem therefore gives
\[
   \frac{A_K(tu)}t
   =2s\int_{\theta^\perp}
      \min\left\{1,\frac{\ell(y)}{ts}\right\}
   \,dy.
\]
The integrand is supported on the bounded projection $K|\theta^\perp$ and
is bounded by one.  Thus, as $t\downarrow0$, dominated convergence yields
\[
   \frac{A_K(tu)}t
   \longrightarrow
   2s\,\operatorname{length}(K|\theta^\perp),
\]
where $K|\theta^\perp$ is the orthogonal projection of $K$.  Since
$K=-K$, this projection is a symmetric interval of length
$2h_K(R\theta)$.  Hence the last limit equals
\[
   4s h_K(R\theta)=4h_K(Ru),
\]
which proves \eqref{eq:convex-symdiff-limit}.  The case $u=0$ is immediate.

For the lower bound, choose $\rho>0$ such that the Euclidean ball of radius
$\rho$ about the origin is contained in $K$.  If
$h=t\theta$ is the shortest representative of a sufficiently small nonzero
torus displacement, then for every $y\in\theta^\perp$ with
$\|y\|_2\le\rho/2$, the interval $I_y$ has length at least
$\sqrt3\rho$.  Thus, when $t$ is sufficiently small, the symmetric
difference on each such line has length $2t$.  Integrating over an interval
of $y$-values of length $\rho$ gives
\[
   A_K(h)\ge2\rho t.
\]

It remains to consider displacements bounded away from zero.  The function
$A_K$ is continuous.  Also, $A_K(h)=0$ only for $h=0$ in $\T^2$.  To see
this, suppose that $A_K(h)=0$.  The two torus sets agree up to a null set;
since they are closures of their interiors and their boundaries have area
zero, they agree as closed sets.  Let $\widetilde h\in\R^2$ be a lift of
$h$.  Then $K+\widetilde h$ is a connected subset of
\[
   q^{-1}(q(K))=\bigcup_{z\in\Z^2}(K+z),
\]
where $q:\R^2\to\T^2$ is the quotient map.  The sets $K+z$ are pairwise
separated because $K$ is compactly contained in a fundamental square.  Hence
$K+\widetilde h$ is contained in one translate $K+z$.  The two sets have the
same area, so they are equal.  Since $K$ is bounded, this forces
$\widetilde h=z$, and therefore $h=0$ on the torus.  Compactness now gives a
positive lower bound for $A_K(h)/|h|_{\T,2}$ away from zero.  Combining this
with the local estimate proves \eqref{eq:convex-global-lower}.
\end{proof}

We now prove the planar Poisson theorem.

\begin{proof}[Proof of Theorem~\ref{thm:planar-convex-twins}]
Write $T_n=T_n(K)$.  For two vertices at displacement $h$, they are
adjacent twins exactly when $h\in K$ and none of the remaining $n-2$
sample points lies in $D_K(h)$.  Hence
\begin{equation}
\label{eq:convex-twin-expectation}
   \E T_n
   =\binom n2\int_K(1-A_K(h))^{n-2}\,dh.
\end{equation}
After the change of variables $u=nh$,
\[
   \E T_n
   =\frac{n(n-1)}{2n^2}
     \int_{nK}
       \left(1-A_K\left(\frac un\right)\right)^{n-2}
     \,du.
\]
For every fixed $u$, Lemma~\ref{lem:convex-symdiff} gives
\[
   nA_K(u/n)\longrightarrow4h_K(Ru).
\]
Since the origin lies in the interior of $K$, every fixed $u$ belongs to
$nK$ for all sufficiently large $n$.  For all sufficiently large $n$, the global lower bound in
\eqref{eq:convex-global-lower} gives
\[
   \1_{\{u\in nK\}}
   \left(1-A_K\left(\frac un\right)\right)^{n-2}
   \le e^{-c\|u\|_2}
\]
for a constant $c>0$.  This is integrable on $\R^2$.  Therefore
\begin{equation}
\label{eq:convex-twin-mean}
   \E T_n\longrightarrow
   \frac12\int_{\R^2}e^{-4h_K(Ru)}\,du.
\end{equation}

We evaluate this integral before proving the Poisson limit.  Since rotation
preserves area,
\[
   \int_{\R^2}e^{-4h_K(Ru)}\,du
   =\int_{\R^2}e^{-4h_K(u)}\,du.
\]
The polar body satisfies
\[
   K^\circ=\{u:h_K(u)\le1\},
\]
so $h_K$ is the Minkowski functional of $K^\circ$.  Consequently,
\[
   \area\{u:h_K(u)\le t\}=t^2\area(K^\circ).
\]
Integration with respect to these level sets gives
\begin{equation}
\label{eq:polar-integral}
\begin{aligned}
   \int_{\R^2}e^{-4h_K(u)}\,du
   &=2\area(K^\circ)\int_0^\infty te^{-4t}\,dt\\
   &=\frac{\area(K^\circ)}8.
\end{aligned}
\end{equation}
Thus the limit in \eqref{eq:convex-twin-mean} is
\[
   \lambda_K:=\frac{\area(K^\circ)}{16}.
\]

We prove convergence to $\Poi(\lambda_K)$ by factorial moments.  Fix
$m\ge1$ and expand $(T_n)_m$ over ordered $m$-tuples of distinct unordered
vertex pairs.  We first consider tuples with overlapping pairs.  Let such a
tuple span a graph $F$ on $v$ vertices with $c$ connected components.  Since
at least one component contains two selected pairs,
\[
   v>2c.
\]
Choose a spanning forest of $F$.  For each forest edge $e$, let $h_e$ be its
torus displacement.  If every selected pair consists of adjacent twins, all
sample points outside the $v$ displayed vertices must avoid the union of the
disagreement regions associated with the forest edges.  By
\eqref{eq:convex-global-lower}, conditional on the displayed vertices this
has probability at most
\[
   \exp\left\{-cn\max_e|h_e|_{\T,2}\right\}
\]
for a constant $c>0$.  Exposing one root in each component and then the
$v-c$ forest displacements gives
\[
   \Prob(\text{all selected pairs are twins})
   =O\bigl(n^{-2(v-c)}\bigr).
\]
There are $O(n^v)$ choices of the displayed labels.  The contribution of
this overlap type is therefore
\[
   O(n^{2c-v})=o(1).
\]
There are only finitely many overlap types for fixed $m$, so their total
contribution is $o(1)$.

It remains to consider vertex-disjoint pairs.  By symmetry, consider the
fixed matching
\[
   \{1,2\},\{3,4\},\ldots,\{2m-1,2m\}.
\]
Write
\[
   Y_a=X_{2a-1},
   \qquad
   H_a=X_{2a}-X_{2a-1},
   \qquad 1\le a\le m.
\]
For given $Y=(Y_1,\ldots,Y_m)$ and $H=(H_1,\ldots,H_m)$, let
\[
   \Delta_a(Y,H)=Y_a+D_K(H_a),
   \qquad
   V(Y,H)=\area\left(\bigcup_{a=1}^m\Delta_a(Y,H)\right).
\]
Let $\mathcal C(Y,H)$ denote the deterministic compatibility condition that
every displayed endpoint has the same adjacency to the two endpoints of
each other displayed pair.  If $P_{m,n}$ is the probability that the fixed
matching consists of adjacent twin pairs, then
\begin{equation}
\label{eq:convex-fixed-matching}
   P_{m,n}
   =\int_{K^m}\int_{(\T^2)^m}
      \1_{\mathcal C(Y,H)}(1-V(Y,H))^{n-2m}
   \,dY\,dH.
\end{equation}

Set $H_a=U_a/n$.  For fixed $U=(U_1,\ldots,U_m)$, put
\[
   D_{a,n}=D_K(U_a/n),
   \qquad
   a_{a,n}=\area(D_{a,n}),
   \qquad
   S_n(U)=\sum_{a=1}^m a_{a,n},
\]
and write
\[
   V_n(Y,U)=\area\left(\bigcup_{a=1}^m(Y_a+D_{a,n})\right).
\]
For measurable $A,B\subseteq\T^2$, Fubini's theorem gives
\begin{equation}
\label{eq:translate-overlap-average}
\begin{aligned}
&\int_{\T^2\times\T^2}
   \area\bigl((y+A)\cap(z+B)\bigr)\,dy\,dz\\
&\qquad=\area(A)\area(B).
\end{aligned}
\end{equation}
Since
\[
   0\le S_n(U)-V_n(Y,U)
   \le\sum_{a<b}
      \area\bigl((Y_a+D_{a,n})\cap(Y_b+D_{b,n})\bigr),
\]
we obtain
\begin{equation}
\label{eq:union-L1}
   \int_{(\T^2)^m}\bigl(S_n(U)-V_n(Y,U)\bigr)\,dY
   \le\sum_{a<b}a_{a,n}a_{b,n}.
\end{equation}
By Lemma~\ref{lem:convex-symdiff}, for fixed $U$,
\[
   na_{a,n}\longrightarrow4h_K(RU_a).
\]
It follows from \eqref{eq:union-L1} that
\begin{equation}
\label{eq:union-volume-limit}
   nV_n(Y,U)\longrightarrow
   4\sum_{a=1}^m h_K(RU_a)
\end{equation}
in $L^1(dY)$.  Moreover, $0\le nV_n(Y,U)\le nS_n(U)$, and the latter is
bounded for fixed $U$.  Uniformly in $Y$,
\[
   (n-2m)\log(1-V_n(Y,U))=-nV_n(Y,U)+o_U(1).
\]
Since the exponential is Lipschitz on bounded intervals,
\eqref{eq:union-volume-limit} therefore implies
\begin{equation}
\label{eq:avoidance-average-limit}
\begin{aligned}
&\int_{(\T^2)^m}(1-V_n(Y,U))^{n-2m}\,dY\\
&\qquad\longrightarrow
   \exp\left\{-4\sum_{a=1}^m h_K(RU_a)\right\}.
\end{aligned}
\end{equation}

The displayed compatibility condition has no effect on this limit.  Indeed,
for $a\ne b$, either endpoint of pair $b$ fails compatibility with pair $a$
only if it lies in $Y_a+D_{a,n}$.  Averaging over the locations of the pairs
and using a union bound gives
\[
   \int_{(\T^2)^m}
      \bigl(1-\1_{\mathcal C(Y,U/n)}\bigr)\,dY
   \le2(m-1)\sum_{a=1}^m a_{a,n}
   =O_U(n^{-1}).
\]
Combining this with \eqref{eq:avoidance-average-limit}, we get
\begin{equation}
\label{eq:compatible-average-limit}
\begin{aligned}
&\int_{(\T^2)^m}
   \1_{\mathcal C(Y,U/n)}
   (1-V_n(Y,U))^{n-2m}\,dY\\
&\qquad\longrightarrow
   \exp\left\{-4\sum_{a=1}^m h_K(RU_a)\right\}.
\end{aligned}
\end{equation}

The global lower bound gives
\[
   V_n(Y,U)\ge\max_a a_{a,n}
   \ge\frac{c_K}{n}\max_a\|U_a\|_2
\]
whenever $U_a\in nK$.  For all sufficiently large $n$, the integrand after
the change of variables is therefore dominated by
\[
   \exp\left\{-c\max_a\|U_a\|_2\right\},
\]
which is integrable on $(\R^2)^m$.  Since every fixed $U_a$ belongs to $nK$
for all sufficiently large $n$, dominated convergence in
\eqref{eq:convex-fixed-matching} yields
\begin{align*}
   n^{2m}P_{m,n}
   &\longrightarrow
   \int_{(\R^2)^m}
      \exp\left\{-4\sum_{a=1}^m h_K(RU_a)\right\}
   \,dU_1\cdots dU_m\\
   &=\left(\frac{\area(K^\circ)}8\right)^m,
\end{align*}
where we used \eqref{eq:polar-integral}.

The number of ordered $m$-tuples of vertex-disjoint unordered pairs is
\[
   \frac{(n)_{2m}}{2^m}.
\]
Their contribution to the $m$th factorial moment therefore tends to
\[
   \frac1{2^m}
   \left(\frac{\area(K^\circ)}8\right)^m
   =\left(\frac{\area(K^\circ)}{16}\right)^m
   =\lambda_K^m.
\]
Together with the negligible overlapping contribution, this proves
\[
   \E[(T_n)_m]\longrightarrow\lambda_K^m
   \qquad(m=1,2,\ldots).
\]
These are the factorial moments of $\Poi(\lambda_K)$, proving
\eqref{eq:convex-twin-poisson}.

The sentence \eqref{eq:twin-sentence} holds exactly when $T_n\ge1$.  Its
limiting probability is therefore $1-e^{-\lambda_K}$, which lies strictly
between zero and one.

For $K=[-r,r]^2$, the polar body is
\[
   K^\circ=\{u:r(|u_1|+|u_2|)\le1\},
\]
which has area $2/r^2$.  For $K=rB_2$, the polar body is
$r^{-1}B_2$, which has area $\pi/r^2$.  The two stated special cases follow.
\end{proof}

\begin{remark}
The Euclidean corollary gives a direct non-trivial first-order limit for
every $0<r<1/2$.  For sufficiently small fixed radius, the result of
Haber, Hershko and M\"uller \cite{HHM2022} is stronger in a different
direction: it shows that the full first-order convergence law fails.
\end{remark}


\subsection{The cube model in arbitrary dimension}

For $h\in\T^D$, put
\[
   A_r(h):=\vol(B_r(0)\symdiff B_r(h)).
\]
Whenever coordinates of $h$ are written as real numbers, we use the
representatives in $[-1/2,1/2)$.

\begin{lemma}
\label{lem:symdiff}
Fix $D\ge1$ and $0<r<1/2$.  As $h\to0$ in $\T^D$,
\begin{equation}
\label{eq:symdiff-expansion}
   A_r(h)=2^D r^{D-1}\sum_{i=1}^D |h_i|+O(\|h\|_1^2).
\end{equation}
Moreover, there are constants $\eta_0,c_0,C_0>0$ such that
\[
   c_0\|h\|_1\le A_r(h)\le C_0\|h\|_1
\]
whenever $\|h\|_\infty\le\eta_0$.  For every $\eta>0$ there is $c_\eta>0$
such that
\[
   A_r(h)\ge c_\eta
   \qquad\text{whenever }d_\infty(0,h)\ge\eta.
\]
In particular, there is a constant $c_1>0$ such that
\begin{equation}
\label{eq:symdiff-global-lower}
   A_r(h)\ge c_1\|h\|_\infty
   \qquad\text{for all }h\in\T^D.
\end{equation}
\end{lemma}

\begin{proof}
For $\|h\|_\infty<\min(r,1/2-r)$ there is no wrap-around ambiguity, and the
two balls are axis-parallel cubes of side length $2r$.  Their intersection
has volume
\[
   \prod_{i=1}^D(2r-|h_i|).
\]
Therefore
\begin{align*}
   A_r(h)
   &=2\left((2r)^D-\prod_{i=1}^D(2r-|h_i|)\right)\\
   &=2^D r^{D-1}\sum_{i=1}^D |h_i|+O(\|h\|_1^2).
\end{align*}
This proves \eqref{eq:symdiff-expansion}, and the local two-sided estimate
follows by taking $\eta_0$ small enough.

For the compactness assertion, note that $A_r(h)$ is continuous in $h$.  Since
$r<1/2$, the two torus cubes $B_r(0)$ and $B_r(h)$ agree up to null sets only
when $h=0$.  Hence $A_r(h)$ is bounded away from zero on every compact set
$\{h:d_\infty(0,h)\ge\eta\}$.

Finally, combine the local lower bound near zero with the preceding compactness
bound away from zero.  Since $\|h\|_\infty\le1/2$ for our representatives,
after decreasing the constant if necessary we obtain
\eqref{eq:symdiff-global-lower}.
\end{proof}

\begin{proof}[Expectation asymptotic]
For two vertices at displacement $h$, they are adjacent twins exactly when
$d_\infty(0,h)\le r$ and no other sample point lies in
$B_r(0)\symdiff B_r(h)$.  Hence
\begin{equation}
\label{eq:twin-expect-formula}
   \E T_{n,D}(r)
   =
   \binom n2
   \int_{\T^D}
      \1_{\{d_\infty(0,h)\le r\}}
      (1-A_r(h))^{n-2}\,dh.
\end{equation}

Fix a small $\eta>0$.  By Lemma~\ref{lem:symdiff}, the part of the integral
with $d_\infty(0,h)\ge\eta$ is exponentially small in $n$.  On
$d_\infty(0,h)<\eta$, make the change of variables $u=nh$.  Then
\[
   nA_r(u/n)\longrightarrow 2^D r^{D-1}\|u\|_1.
\]
The local lower bound in Lemma~\ref{lem:symdiff} gives an integrable
exponential domination, so dominated convergence yields
\begin{align*}
   &\int_{\T^D}
      \1_{\{d_\infty(0,h)\le r\}}
      (1-A_r(h))^{n-2}\,dh \\
   &\hspace{3em}\sim
   n^{-D}
   \int_{\R^D}
      e^{-2^D r^{D-1}\|u\|_1}\,du  \\
   &\hspace{3em}=
   n^{-D}
   \left(\frac{2}{2^D r^{D-1}}\right)^D
   =
   n^{-D}2^{-D(D-1)}r^{-D(D-1)} .
\end{align*}
Multiplying by $\binom n2$ proves
\[
   \E T_{n,D}(r)
   \sim
   2^{D-1-D^2}r^{-D(D-1)}n^{2-D},
\]
as claimed.
\end{proof}

\begin{proof}[Remaining assertions]
When $D=2$, the graph $G_2(n;r)$ is the convex connection model with
$K=[-r,r]^2$.  The Poisson convergence in
\eqref{eq:poisson-twins} therefore follows from
Theorem~\ref{thm:planar-convex-twins}.  Since the sentence
\eqref{eq:twin-sentence} holds exactly when an adjacent twin pair exists, its
limiting probability is
$
   1-e^{-1/(8r^2)},
$
which lies strictly between zero and one.

Now suppose that $D\ge3$.  The expectation estimate and Markov's inequality give
$
   T_{n,D}(r)\longrightarrow0
$
in probability.
\end{proof}

\begin{remark}
For $D\ge3$, the disappearance of adjacent twins only shows that this
particular fixed-radius obstruction is absent; it does not imply a zero-one
law.
\end{remark}

\section{Thin common-neighborhood cells}
\label{sec:higher-arity}

Fix $D\ge3$ and $0<r<1/2$, and put
\[
   g=1-2r,
   \qquad
   q=\left\lfloor\frac1g\right\rfloor+1,
   \qquad
   M_D=2^{D-2}+1.
\]
Then $q\ge2$ and $qg>1$.  For vertices $x_1,\ldots,x_q$, let
$\Theta_q(x_1,\ldots,x_q)$ say that the roots are distinct, have exactly
$M_D$ common neighbors, and that these common neighbors are pairwise
non-adjacent:
\[
\begin{aligned}
\Theta_q(x_1,\ldots,x_q):={}&
\left(\bigwedge_{1\le a<b\le q}x_a\ne x_b\right)
\wedge
\exists z_1\cdots\exists z_{M_D}\bigg[\\
&\bigwedge_{i=1}^{M_D}\bigwedge_{a=1}^{q} E(z_i,x_a)\\
&\wedge
  \bigwedge_{1\le i<j\le M_D}
  (z_i\ne z_j\wedge \neg E(z_i,z_j))\\
&\wedge
  \forall z\left(
       \left(\bigwedge_{a=1}^{q}E(z,x_a)\right)
       \to \bigvee_{i=1}^{M_D}z=z_i
  \right)
\bigg].
\end{aligned}
\]
Because the graphs are loopless, none of the roots can be among their common
neighbors.

For $t\in\T$, write
\[
   I_r(t)=\{u\in\T:\|u-t\|_{\T}\le r\}.
\]
For $\bar t=(t_1,\ldots,t_q)\in\T^q$, set
\[
   J(\bar t)=\bigcap_{a=1}^q I_r(t_a),
   \qquad
   \ell(\bar t)=|J(\bar t)|.
\]
For a Borel set $J\subseteq\T$, let $\kappa(J)$ be the least number of
subsets of torus diameter at most $r$ needed to cover $J$, with
$\kappa(\varnothing)=0$.  Since $J(\bar t)\subseteq I_r(t_1)$ and
$2r<1$, every non-empty $J(\bar t)$ satisfies $\kappa(J(\bar t))\le2$.

\begin{lemma}
\label{lem:one-dimensional-thin}
There are constants $C,\eta>0$, depending only on $r$ and $q$, such that for
all sufficiently small $s>0$,
\[
   \Prob(0<\ell(\bar T)\le s)\le Cs
\]
and
\[
   \Prob(0<\ell(\bar T)\le s,\ \kappa(J(\bar T))=2)\le Cs^2,
\]
where $\bar T=(T_1,\ldots,T_q)$ is uniform on $\T^q$.  Moreover,
\[
   \Prob(\ell(\bar T)>r+\eta)>0.
\]
\end{lemma}

\begin{proof}
The complement of $I_r(t)$ is an open interval of length $g=1-2r$.
Let $s_a$ be the left endpoint of the complement of $I_r(t_a)$.  Outside a
null set the $s_a$ are distinct.  On a fixed cyclic-ordering chart, after
separating the rotation variable, write
$
   0=s_1<s_2<\cdots<s_q<1
$
and put
\[
   d_i=s_{i+1}-s_i\quad(1\le i\le q),
   \qquad s_{q+1}=s_1+1.
\]
Thus $d_i>0$ and $\sum_i d_i=1$.  The components of
$J(\bar t)$ have lengths
\[
   (d_i-g)_+,
   \qquad 1\le i\le q,
\]
so
\begin{equation}
\label{eq:ell-spacing}
   \ell(\bar t)=\sum_{i=1}^q(d_i-g)_+.
\end{equation}

There are finitely many cyclic-ordering charts, each with bounded Jacobian.
If $0<\ell\le s$, then for some $i$,
\[
   g<d_i\le g+s.
\]
On each spacing simplex this is a union of finitely many slabs of thickness
$s$, and therefore has measure $O(s)$.

Suppose now that $0<\ell\le s<r$ and $\kappa(J)=2$.  A single component of
$J$ would have torus diameter at most $\ell<r$ and hence would be covered by
one set of diameter at most $r$.  Thus $J$ has at least two non-empty
components.  By \eqref{eq:ell-spacing}, two distinct spacings satisfy
\[
   g<d_i\le g+s,
   \qquad
   g<d_j\le g+s.
\]
The intersection of the corresponding two slabs has measure $O(s^2)$.
Summing over the finitely many charts and pairs $i\ne j$ proves the second
bound.

Finally, when all the centers $t_a$ lie in a sufficiently short arc, their
common intersection has length arbitrarily close to $2r$.  Since $2r>r$,
this gives some $\eta>0$ and a positive-measure set on which
$\ell(\bar t)>r+\eta$.
\end{proof}

\begin{lemma}
\label{lem:product-tail}
Let $\bar X=(X_1,\ldots,X_q)$ be uniform on $(\T^D)^q$.  For
$1\le j\le D$, set
\[
   J_j(\bar X)=\bigcap_{a=1}^q I_r(X_{a,j}),
   \qquad
   \ell_j=|J_j(\bar X)|,
   \qquad
   \kappa_j=\kappa(J_j(\bar X)).
\]
There is a constant $C<\infty$ such that, for all sufficiently small $t>0$,
\[
   \Prob\left(
      0<\prod_{j=1}^D\ell_j\le t
      \text{ and }
      \#\{j:\kappa_j=2\}\ge D-1
   \right)
   \le Ct.
\]
\end{lemma}

\begin{proof}
The pairs $(\ell_j,\kappa_j)$ are independent over the coordinates.  After
increasing $C$ if necessary, Lemma~\ref{lem:one-dimensional-thin} gives, for
all $s>0$,
\[
   \Prob(0<\ell_j\le s)\le Cs,
   \qquad
   \Prob(0<\ell_j\le s,\ \kappa_j=2)\le Cs^2.
\]
The quadratic bound implies
\begin{equation}
\label{eq:inverse-moment}
   \E[\ell_j^{-1};\ \ell_j>0,\ \kappa_j=2]<\infty.
\end{equation}
Indeed,
\[
\begin{aligned}
\E[\ell_j^{-1};\ \ell_j>0,\ \kappa_j=2]
&\le 1+\int_1^\infty
   \Prob(0<\ell_j<t^{-1},\ \kappa_j=2)\,dt\\
&\le 1+C\int_1^\infty t^{-2}\,dt<\infty.
\end{aligned}
\]

Fix $S\subseteq\{1,\ldots,D\}$ with $|S|=D-1$, and let $j_0$ be the
remaining coordinate.  Conditioning on $(\ell_j,\kappa_j)_{j\in S}$ and
using the linear tail estimate for $\ell_{j_0}$ gives
\[
\begin{aligned}
&\Prob\left(
   0<\prod_{j=1}^D\ell_j\le t
   \text{ and }\kappa_j=2\text{ for every }j\in S
\right)\\
&\qquad\le
Ct\,\E\left[
   \prod_{j\in S}\ell_j^{-1};\
   \ell_j>0,\ \kappa_j=2\text{ for every }j\in S
\right]
\le C't,
\end{aligned}
\]
where independence and \eqref{eq:inverse-moment} are used in the last step.
The event in the statement is contained in the union of these events over
the $D$ choices of $S$.
\end{proof}

\begin{lemma}
\label{lem:product-packing}
Let $J_1,\ldots,J_D\subseteq\T$ be Borel sets, each contained in an arc of
length $2r$, and put
$
   K=J_1\times\cdots\times J_D.
$ 
If at most $s$ of the sets $J_j$ require two subsets of torus diameter at
most $r$ to cover them, then $K$ contains at most $2^s$ points whose
pairwise $L^\infty$-distances are all greater than $r$.
\end{lemma}

\begin{proof}
Cover each $J_j$ by one or two sets of torus diameter at most $r$.  Their
products cover $K$ by at most $2^s$ sets of $L^\infty$-diameter at most $r$.
A pairwise $r$-separated subset of $K$ contains at most one point from each
product set.
\end{proof}

\begin{proof}[Proof of Theorem~\ref{thm:higher-arity-thin-cells}]
Let
\[
   K(\bar X)=\bigcap_{a=1}^q B_r(X_a).
\]
By the product form of $L^\infty$ balls,
\[
   K(\bar X)=J_1(\bar X)\times\cdots\times J_D(\bar X),
   \qquad
   V(\bar X):=\vol K(\bar X)=\prod_{j=1}^D\ell_j.
\]
Conditioned on the roots, the number of non-root sample points in
$K(\bar X)$ has distribution $\operatorname{Bin}(n-q,V(\bar X))$.

We first prove the upper bound.  If $\Theta_q(1,\ldots,q)$ holds, then
$K(\bar X)$ contains $M_D=2^{D-2}+1$ sample points whose pairwise
$L^\infty$-distances are greater than $r$.  By
Lemma~\ref{lem:product-packing}, at least $D-1$ coordinates satisfy
$\kappa_j=2$.  For fixed $M_D$, uniformly in $0\le V\le1$,
\[
   \Prob(\operatorname{Bin}(n-q,V)=M_D)
   \le C(nV)^{M_D}e^{-cnV}
\]
for constants $C,c>0$.  Hence
\[
\begin{aligned}
\Prob\bigl(G_D(n;r)\models\Theta_q(1,\ldots,q)\bigr)
\le C\E\left[
   \1_{\{\#\{j:\kappa_j=2\}\ge D-1\}}
   (nV)^{M_D}e^{-cnV}
\right].
\end{aligned}
\]

Choose $t_0>0$ so that Lemma~\ref{lem:product-tail} applies on
$(0,t_0]$.  Decompose the positive values of $V$ into
\[
   0<V\le\frac2n,
   \qquad
   \frac{2^m}{n}<V\le\frac{2^{m+1}}n
   \quad(m\ge1).
\]
For $2^{m+1}/n\le t_0$, Lemma~\ref{lem:product-tail} bounds the probability
of the relevant event with $V\le2^{m+1}/n$ by $C2^m/n$, while on the same
interval
\[
   (nV)^{M_D}e^{-cnV}
   \le C(2^m)^{M_D}e^{-c2^m}.
\]
The contribution of all these intervals is therefore at most
\[
   \frac Cn\sum_{m\ge0}2^m(2^m)^{M_D}e^{-c2^m}
   =O(n^{-1}).
\]
The part with $V\ge t_0/2$ is exponentially small.  This proves the upper
bound.

For the lower bound, choose coordinate $1$ to be thin.  On a cyclic spacing
chart from Lemma~\ref{lem:one-dimensional-thin}, impose
\[
   d_1=g+u,
   \qquad 0<u<\varepsilon,
\]
and require $d_i<g$ for $2\le i\le q$.  Such a chart exists because
$qg>1$: at $u=0$, the point
\[
   d_2=\cdots=d_q=\frac{1-g}{q-1}
\]
lies strictly between zero and $g$.  Thus, for all sufficiently small $u$,
the remaining spacings may be chosen in a fixed compact set of positive
relative measure around this point.  On this chart $J_1$ is a single
interval of length $u$.

In each coordinate $2,\ldots,D$, choose the $q$ centers in a fixed compact
positive-measure set on which they lie in a short common arc.  After
shrinking this set, there is an $\eta>0$ such that throughout the chart
\[
   |J_j|>r+\eta,
   \qquad 2\le j\le D.
\]
Now restrict
\[
   u=\frac{s}{n},
   \qquad 1\le s\le2.
\]
Then
\[
   V(\bar X)=\frac{s}{n}A(\mathbf y),
\]
where $A(\mathbf y)$ is bounded above and below by positive constants on the
chosen compact chart.  It follows that the conditional probability of
having exactly $M_D$ non-root sample points in $K(\bar X)$ is bounded below
by a positive constant, uniformly on the chart.

Choose
\[
   0<\delta<\frac13\min\{\eta,1-2r\}.
\]
Every long interval $J_j$, $2\le j\le D$, contains two small subintervals
such that the torus distance between any point of one and any point of the
other is greater than $r$.  Indeed, inside a lift of $J_j$ one may take one
subinterval near the left endpoint and a second beginning at distance
$r+2\delta$; the assumptions on $\delta$ keep both subintervals inside
$J_j$ and ensure that their ordinary separation is the shorter torus
separation.  Taking products over the $D-1$ long coordinates gives
$2^{D-1}$ subboxes of $K(\bar X)$ which are pairwise separated by more than
$r$ in $L^\infty$-distance.  Since
\[
   M_D=2^{D-2}+1\le2^{D-1},
\]
choose $M_D$ of them.  Conditional on exactly $M_D$ points falling in
$K(\bar X)$, the probability that one point falls in each chosen subbox is
bounded below by a positive constant.

On the fixed spacing chart the root density is bounded above and below on
the chosen compact set.  The restriction $1/n\le u\le2/n$ contributes a
factor of order $1/n$, while all other variables range over sets of fixed
positive measure.  Hence
\[
   \Prob\bigl(G_D(n;r)\models\Theta_q(1,\ldots,q)\bigr)
   \ge \frac{c_{D,r}}n.
\]
Together with the upper bound, this proves the theorem.
\end{proof}

\begin{remark}
The order $1/n$ is the natural scale for an extension formula with one free
selector.  In the Euclidean two-dimensional model, the arithmetic
interpretation of \cite{HHM2022} rests on much stronger information: uniform
joint Poisson estimates and concentration for families of extensions.  The
present theorem is only a marginal estimate.  Establishing an appropriate
joint extension theorem for the torus $L^\infty$ model remains a separate
problem.
\end{remark}

\begin{remark}
If $r_n\ge1/2$ eventually, then for every $x,y\in\T^D$ one has
\[
   d_\infty(x,y)\le1/2\le r_n.
\]
Hence $G_D(n;r_n)=K_n$ deterministically for all sufficiently large $n$.
Lemma~\ref{lem:empty-complete-matching} then implies that every first-order
sentence has an eventual truth value, so the first-order zero-one law holds in
this regime.
\end{remark}


\begin{thebibliography}{99}

\bibitem{Fagin1976}
R. Fagin.
\newblock Probabilities on finite models.
\newblock \emph{Journal of Symbolic Logic}, 41(1):50--58, 1976.

\bibitem{Gilbert1961}
E. N. Gilbert.
\newblock Random plane networks.
\newblock \emph{Journal of the Society for Industrial and Applied Mathematics},
9(4):533--543, 1961.

\bibitem{Glebskii1969}
Yu. V. Glebskii, D. I. Kogan, M. I. Liogonkii, and V. A. Talanov.
\newblock Range and degree of realizability of formulas in the restricted
predicate calculus.
\newblock \emph{Cybernetics}, 5:142--154, 1969.

\bibitem{HHM2022}
S. Haber, T. Hershko, and T. M\"uller.
\newblock Disproving the limit law for random geometric graphs.
\newblock Preprint, 2022.
\newblock \url{https://u.cs.biu.ac.il/~simi/onlineavailablepapers/nolimitlrggs.pdf}.


\bibitem{LangharstRoysdonZvavitch2022}
D. Langharst, M. Roysdon, and A. Zvavitch.
\newblock General measure extensions of projection bodies.
\newblock \emph{Proceedings of the London Mathematical Society},
125(5):1083--1129, 2022.
\newblock \url{https://doi.org/10.1112/plms.12477}.

\bibitem{Libkin2004}
L. Libkin.
\newblock \emph{Elements of Finite Model Theory}.
\newblock Springer, 2004.

\bibitem{McColm1999}
G. L. McColm.
\newblock First order zero-one laws for random graphs on the circle.
\newblock \emph{Random Structures \& Algorithms}, 14(3):239--266, 1999.

\bibitem{Penrose2003}
M. Penrose.
\newblock \emph{Random Geometric Graphs}.
\newblock Oxford Studies in Probability, vol. 5. Oxford University Press,
2003.

\bibitem{ShelahSpencer1988}
S. Shelah and J. Spencer.
\newblock Zero-one laws for sparse random graphs.
\newblock \emph{Journal of the American Mathematical Society}, 1(1):97--115,
1988.

\bibitem{Spencer2001}
J. Spencer.
\newblock \emph{The Strange Logic of Random Graphs}.
\newblock Algorithms and Combinatorics, vol. 22. Springer, 2001.

\bibitem{SpencerAgarwal2006}
J. Spencer and A. Agarwal.
\newblock Undecidable statements and the zero-one law in random geometric
graphs.
\newblock Unpublished manuscript, 2006.

\end{thebibliography}
\end{document}